\documentclass[12pt]{article}
\usepackage[utf8]{inputenc}
\usepackage[T1]{fontenc}
\usepackage[top=25mm, bottom=25mm, left=25mm, right=25mm, marginparwidth=28mm]{geometry}
\usepackage{amsfonts,amsmath,amsthm,amssymb}
\usepackage{hyperref,enumitem,enumerate}
\usepackage{xcolor}
\hypersetup{
    colorlinks=true,
    urlcolor=black,
}
\usepackage{graphics}
\usepackage{pstricks,pst-node,pst-arrow}
\PassOptionsToPackage{usenames,dvipsnames,svgnames}{xcolor}  
\usepackage{tikz}
\usetikzlibrary{arrows,positioning,automata, fit}
\usetikzlibrary{arrows.meta}
\usepackage{float}
\usetikzlibrary{graphs,graphs.standard,quotes}
\usetikzlibrary{matrix, positioning, fit}

\usepackage{xifthen}

\usepackage[normalem]{ulem}
\usepackage{xcolor}
\newcommand\redsout{\bgroup\markoverwith{\textcolor{magenta}{\rule[0.5ex]{4pt}{0.8pt}}}\ULon}
\newcommand\bluesout{\bgroup\markoverwith{\textcolor{cyan}{\rule[0.5ex]{4pt}{0.8pt}}}\ULon}
\newcommand\greensout{\bgroup\markoverwith{\textcolor{green}{\rule[0.5ex]{4pt}{0.8pt}}}\ULon}
\newcommand\orangesout{\bgroup\markoverwith{\textcolor{orange}{\rule[0.5ex]{4pt}{0.8pt}}}\ULon}

\def\ArgOne#1{\arrowcolor#1,\relax}
\def\arrowcolor#1,#2,#3\relax{#1}

\def\ArgTwo#1{\arrownumber#1,\relax}
\def\arrownumber#1,#2,#3\relax{#2}

\def\ArrowColor#1{%
    \ArgOne{#1}%
}%

\def\ArrowNumber#1{%
    \ArgTwo{#1}%
}%

\tikzset{pics/.cd,
  pic InnerGraph/.style 2 args={code={
    \def\innerradius{1.5}
    \def\radius{3}
    \def\bendangle{5}

    \foreach \v/\lab in {1, 2, 3, 4, 5, 6, 7, 8}{  
      \node[inner sep=0pt,minimum size=12pt] (n1\v) at (45-\v*45:\innerradius cm) {};
    } 
    \foreach \v/\vnext in {1/4, 8/5}{
      \draw [#1, -{Stealth[length=3mm, width=2mm]}] (n1\v) to[bend right=\bendangle] (n1\vnext); 
      \draw [#1, -{Stealth[length=3mm, width=2mm]}] (n1\vnext) to[bend left=\bendangle] (n1\v); 
    }
    \foreach \v/\vnext in {2/7, 3/6}{
      \draw [#2, -{Stealth[length=3mm, width=2mm]}] (n1\v) to[bend right=\bendangle] (n1\vnext); 
      \draw [#2, -{Stealth[length=3mm, width=2mm]}] (n1\vnext) to[bend left=\bendangle] (n1\v); 
    }
    \node[circle,draw=black!80, line width=0.5mm, inner sep=0pt, fit={(n11) (n12) (n13) (n14) (n15) (n16) (n17) (n18)}] (n1) {};

  }}
}

\tikzset{pics/.cd,
  pic Eliplse/.style n args={7}{code={
  
    \def\innerradius{1.5}
    \def\radius{3}
    \def\innernodesize{0.15}

    \newcommand{\pairarrows}[5]{
    \def\bendangle{##5}
      \begin{tikzpicture}[overlay]
        \ifthenelse{##4 = 1}{\draw [##1, -{Stealth[length=3mm, width=2mm]}] (##2) to[bend right=\bendangle] (##3)}{};
        \ifthenelse{##4 = 2}{\draw [##1, -{Stealth[length=3mm, width=2mm]}] (##3) to[bend left=\bendangle] (##2)}{};
        \ifthenelse{##4 = 3}{
          \draw [##1, -{Stealth[length=3mm, width=2mm]}] (##2) to[bend right=\bendangle] (##3);
          \draw [##1, -{Stealth[length=3mm, width=2mm]}] (##3) to[bend left=\bendangle] (##2);}{};
      \end{tikzpicture}  
    }

    \node[circle, minimum size=\innernodesize cm, inner sep=0pt, outer sep=0pt] (ne1) {};
    \node[circle, minimum size=\innernodesize cm, inner sep=0pt, outer sep=0pt]  (ne2) [right = 2cm of ne1] {};
    \node[circle, minimum size=\innernodesize cm, inner sep=0pt, outer sep=0pt]  (ne3) [below =0.5 cm of ne1] {};
    \node[circle, minimum size=\innernodesize cm, inner sep=0pt, outer sep=0pt, label={[label distance=1cm]90:#7}]  (ne4) [below =0.5 cm of ne2] {};
    
    \pairarrows{\ArrowColor{#1}}{ne1.south}{ne2.south}{\ArrowNumber{#1}}{8}
    \pairarrows{\ArrowColor{#2}}{ne1.east}{ne2.west}{\ArrowNumber{#2}}{0}
    \pairarrows{\ArrowColor{#3}}{ne1.north}{ne2.north}{\ArrowNumber{#3}}{-8}
    
    \pairarrows{\ArrowColor{#4}}{ne3.south}{ne4.south}{\ArrowNumber{#4}}{8}
    \pairarrows{\ArrowColor{#5}}{ne3.east}{ne4.west}{\ArrowNumber{#5}}{0}
    \pairarrows{\ArrowColor{#6}}{ne3.north}{ne4.north}{\ArrowNumber{#6}}{-8}

    \node[rectangle, rounded corners=0.5cm, draw=black!80, line width=0.5mm, inner sep=10pt, fit={(ne1) (ne2) (ne3) (ne4)}] (n1) {};

  }}
}

\newcommand{\outerpairarrows}[5]{
    \def\bendangle{#5}
      \begin{tikzpicture}[overlay]
        \ifthenelse{#4 = 1}{\draw [#1, -{Stealth[length=3mm, width=2mm]}] (#2) to[bend right=\bendangle] (#3)}{};
        \ifthenelse{#4 = 2}{\draw [#1, -{Stealth[length=3mm, width=2mm]}] (#3) to[bend left=\bendangle] (#2)}{};
        \ifthenelse{#4 = 3}{
          \draw [#1, -{Stealth[length=3mm, width=2mm]}] (#2) to[bend right=\bendangle] (#3);
          \draw [#1, -{Stealth[length=3mm, width=2mm]}] (#3) to[bend left=\bendangle] (#2);}{};
      \end{tikzpicture}  
    }

\tikzset{pics/.cd,
  pic InnerGraph/.style 2 args={code={
    \def\innerradius{1}
    \def\radius{3}
    \def\bendangle{5}

    \foreach \v/\lab in {1, 2, 3, 4, 5, 6, 7, 8}{  
      \node[inner sep=0pt,minimum size=0pt] (n1\v) at (60-\v*45: \innerradius cm) {};
    } 
    \foreach \v/\vnext in {1/4, 8/5}{
      \draw [#1, -{Stealth[length=3mm, width=2mm]}] (n1\v) to[bend right=\bendangle] (n1\vnext); 
      \draw [#1, -{Stealth[length=3mm, width=2mm]}] (n1\vnext) to[bend left=\bendangle] (n1\v); 
    }
    \foreach \v/\vnext in {7/2, 6/3}{
      \draw [#2, -{Stealth[length=3mm, width=2mm]}] (n1\v) to[bend right=\bendangle] (n1\vnext); 
      \draw [#2, -{Stealth[length=3mm, width=2mm]}] (n1\vnext) to[bend left=\bendangle] (n1\v); 
    }
    \node[rectangle, rounded corners=0.8cm, draw=black!80, line width=0.5mm, inner sep=0pt, fit={(n11) (n12) (n13) (n14) (n15) (n16) (n17) (n18)}] (n1) {};

  }}
}

\tikzset{pics/.cd,
  pic InnerGraphBig/.style 2 args={code={
    \def\innerradius{1}

    \def\radius{3}
    \def\bendangle{5}

    \node[] (n11) at (-1.5,1) {};
    \node[] (n12) [right of=n11] {};
    \node[] (n13) [right of=n12] {};
    \node[] (n14) [right of=n13] {};
    \node[] (n15) [below of=n11] {};
    \node[] (n16) [below of=n15] {};
    \node[] (n17) [right of=n16] {};
    \node[] (n18) [right of=n17] {};
    \node[] (n19) [right of=n18] {};
    \node[] (n110) [above of=n19] {};

    \foreach \v/\vnext in {2/5, 3/6, 4/7, 8/10}{
      \draw [#1, -{Stealth[length=3mm, width=2mm]}] (n1\v) to[bend right=\bendangle] (n1\vnext); 
      \draw [#1, -{Stealth[length=3mm, width=2mm]}] (n1\vnext) to[bend left=\bendangle] (n1\v); 
    }
    \foreach \v/\vnext in {5/7, 1/8, 2/9, 3/10}{
      \draw [#2, -{Stealth[length=3mm, width=2mm]}] (n1\v) to[bend right=\bendangle] (n1\vnext); 
      \draw [#2, -{Stealth[length=3mm, width=2mm]}] (n1\vnext) to[bend left=\bendangle] (n1\v); 
    }
    \node[rectangle, rounded corners=0.8cm, draw=black!80, line width=0.5mm, inner sep=0pt, fit={(n11) (n12) (n13) (n14) (n15) (n16) (n17) (n18)}] (n1) {};

  }}
}

\tikzset{pics/.cd,
  pic EmptyBox/.style 2 args={code={
   \def\innerradius{1.5}
    \def\radius{3}
    \def\innernodesize{0.15}
    
    \node[circle, minimum size=\innernodesize cm, inner sep=0pt, outer sep=0pt] (ne1) {};
    \node[circle, minimum size=\innernodesize cm, inner sep=0pt, outer sep=0pt]  (ne2) [right =#1 cm of ne1] {};
    \node[circle, minimum size=\innernodesize cm, inner sep=0pt, outer sep=0pt]  (ne3) [below =#2 cm of ne1] {};
    
    \node[rectangle, rounded corners=0.5cm, draw=black!80, line width=0.5mm, inner sep=10pt, fit={(ne1) (ne2) (ne3)}] (n1) {};
  }}
}

\newtheorem{theorem}{Theorem}
\newtheorem{lemma}[theorem]{Lemma}

\newtheorem{claim}[theorem]{Claim}

\usepackage{mathtools}

\begin{document}

\title{Directed graphs without rainbow triangles\thanks{This work was supported by the National Science Centre grant 2021/42/E/ST1/00193. An extended abstract announcing the results presented in this paper has been published in the Proceedings of Eurocomb'23.}}

\author{
Sebastian Babiński\thanks{Faculty of Mathematics and Computer Science, Jagiellonian University, {\L}ojasiewicza 6, 30-348 Krak\'{o}w, Poland. E-mail: {\tt Sebastian.Babinski@student.uj.edu.pl}.}\and
Andrzej Grzesik\thanks{Faculty of Mathematics and Computer Science, Jagiellonian University, {\L}ojasiewicza 6, 30-348 Krak\'{o}w, Poland. E-mail: {\tt Andrzej.Grzesik@uj.edu.pl}.}\and
Magdalena Prorok\thanks{AGH University, al.~Mickiewicza 30, 30-059 Krakow, Poland. E-mail: {\tt prorok@agh.edu.pl}.}}

\date{}

\maketitle

\begin{abstract}
One of the most fundamental results in graph theory is Mantel's theorem which determines the maximum number of edges in a triangle-free graph of order $n$. 
Recently a colorful variant of this problem has been solved. In such a variant we consider $c$~graphs on a common vertex set, thinking of each graph as edges in a~distinct color, and want to determine the smallest number of edges in each color which guarantees existence of a rainbow triangle. 
Here, we solve the analogous problem for directed graphs without rainbow triangles, either directed or transitive, for any number of colors. 
The constructions and proofs essentially differ for $c=3$ and $c \geq 4$ and the type of the forbidden triangle.
Additionally, we also solve the analogous problem in the setting of oriented graphs. 
\end{abstract}



\section{Introduction}

A cornerstone of the extremal graph theory is Mantel's theorem from 1907 which determines the maximum possible number of edges in a triangle-free graph of a given order. 
Its natural generalization, known as a Tur\'an problem, is to determine the maximum possible number of edges in an $n$-vertex graph not containing a given graph $F$ as a subgraph.
This is an often studied concept in graph theory with many important results, open problems, and generalizations to various other discrete settings.


In a rainbow version of the Tur\'an problem, for a graph $F$ and an integer $c$ we consider $c$ graphs $G_1, G_2, \ldots, G_c$ on the same set of vertices and ask for the maximum possible number of edges in each graph avoiding appearance of a copy of $F$ having at most one edge from each graph. 
In other words, for every $i \in \{1,2,\ldots,c\}$ we color edges of $G_i$ in color $i$ and forbid all copies of $F$ having non-repeated colors, so called rainbow copies. 
Note that if all $G_i$ are exactly the same, then the existence of a rainbow copy of $F$ is equivalent to the existence of a non-colored copy of~$F$, therefore any bound for the rainbow version gives also a bound for the Tur\'an problem. 

When the forbidden graph $F$ is a triangle, it follows from a result of Keevash, Saks, Sudakov and Verstra\"ete~\cite{KSSV04} that for $c \geq 4$ colors the best possible number of edges in each color without having a rainbow triangle is equal to $\frac{1}{4}n^2$. 
This is achieved in the balanced complete bipartite graph (the same in each color) as in Mantel's theorem. 
Surprisingly, Magnant \cite{Mag15} provided a construction showing that for 3 colors the answer is different. 
Recently, Aharoni, DeVos, Gonz\'alez, Montejano and \v S\'amal \cite{ADMMS20}, answering a question of Diwan and Mubayi~\cite{DM07}, proved that for $3$ colors the optimal asymptotic bound is $\left(\frac{26-2\sqrt{7}}{81}\right)n^2 \approx 0.2557n^2$.

Later, Frankl \cite{Fra22} made a conjecture on the optimal bound for the product of the number of edges in each of the 3 colors without having a rainbow triangle. 
This was disproved by Frankl, Gy\H ori, Hel, Lv, Salia, Tomkins, Varga and Zhu \cite{FGHLSTVZ22}. Finally, Falgas-Ravry, Markstr\" om and R\" aty \cite{FMR22} completely determined the triples of the asymptotic number of edges in each color that forces the existence of a rainbow triangle.
Similar problems were also considered for other rainbow structures than triangles, for instance for paths \cite{BG22}, color-critical graphs~\cite{CKLLS22}, or spanning subgraphs \cite{GHMPS22, JK20}, as well as when the minimum degree instead of the number of edges is optimized \cite{FMR23}. 

Here, we consider the problem in the setting of directed graphs and solve it for any number of colors when a transitive triangle or a directed triangle is the forbidden rainbow graph. 

It occurs that for both kinds of triangles and at least 4 colors, the maximum number of edges in each directed graph is attained when each of them is the same directed graph obtained from $K_{\frac{n}{2},\frac{n}{2}}$ by replacing every edge by edges directed in both directions (see Figure~\ref{fig:K_n,n}).
This construction is natural as in each color we have a directed graph that maximizes the number of edges without creating any triangle. 

\begin{figure}[ht]
\centering
\begin{tikzpicture}
\def\radius{5}
\def\baseangle{30}
\def\bendangle{10}
\def\abovecorrection{-0.5}
\def\nodesizebigelipse{0.15}

\pic [local bounding box=A]                  {pic EmptyBox={0.7}{1.3}};
\pic [local bounding box=B, above right = \abovecorrection cm and 3cm of A]   {pic EmptyBox={0.7}{1.3}};

\draw [green, -{Stealth[length=3mm, width=2mm]}] (A) to (B); 
\draw [green, -{Stealth[length=3mm, width=2mm]}] (B) to (A); 

\begin{scope}[transform canvas={yshift=0.5em}]

\draw [red, -{Stealth[length=3mm, width=2mm]}] (A) to[bend left=\bendangle] (B); 

\draw [red, -{Stealth[length=3mm, width=2mm]}] (B) to[bend right=\bendangle] (A); 

\end{scope}

\begin{scope}[transform canvas={yshift=-0.5em}]

\draw [blue, -{Stealth[length=3mm, width=2mm]}] (A) to[bend right=\bendangle] (B); 

\draw [blue, -{Stealth[length=3mm, width=2mm]}] (B) to[bend left=\bendangle] (A); 

\end{scope}

\end{tikzpicture}
\caption{The optimal construction for at least 4 colors and forbidden rainbow directed or transitive triangle.}\label{fig:K_n,n}
\end{figure}
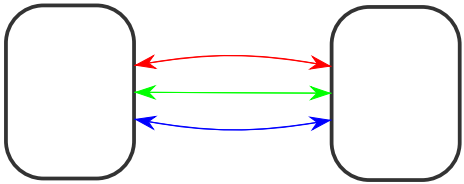

\begin{theorem}\label{cor:4+}
Let $G_1, G_2, \ldots, G_c$ be $c \geq 4$ directed graphs on a common set of $n$ vertices.
If $e(G_i) > \frac{1}{2}n^2$ for every $i \in [c]$, then there exists a rainbow directed triangle and a rainbow transitive triangle.
\end{theorem}

For 3 colors the behavior is completely different. 
If a rainbow directed triangle is forbidden then we prove that the asymptotically optimal construction is obtained by splitting the vertex set into three sets of $n/3$ vertices $A_{1}$, $A_{2}$ and $A_{3}$, and for each $i \in [3]$ letting $A_i$ to contain a complete directed graph in colors different than $i$, and having all edges from $A_1$ to $A_2$, from $A_1$ to $A_3$, and from $A_2$ to $A_3$ (see Figure~\ref{fig:directed3}). 

\begin{figure}[ht]
\centering
\begin{tikzpicture}[scale=0.6]
\def\radius{4}
\def\baseangle{30}

\def\bendangle{10}

\pic [local bounding box=A1] at (\baseangle-1*120:\radius cm) {pic InnerGraph={green}{red}};
\pic [local bounding box=A2] at (\baseangle-2*120:\radius cm) {pic InnerGraph={red}{blue}};
\pic [local bounding box=A3] at (\baseangle-3*120:\radius cm) {pic InnerGraph={blue}{green}};

\draw [red, -{Stealth[length=3mm, width=2mm]}, shorten >=1.05cm, shorten <=1.05cm] (A1.center) to[bend left=\bendangle] (A2.center); 
\draw [blue, -{Stealth[length=3mm, width=2mm]}, shorten >=1.05cm, shorten <=1.05cm] (A1.center) to[bend right=\bendangle] (A2.center); 
\draw [green, -{Stealth[length=3mm, width=2mm]}] (A1) to (A2); 

\draw [red, -{Stealth[length=3mm, width=2mm]}] (A2) to[bend left=\bendangle] (A3); 
\draw [blue, -{Stealth[length=3mm, width=2mm]}] (A2) to[bend right=\bendangle] (A3); 
\draw [green, -{Stealth[length=3mm, width=2mm]}] (A2) to (A3); 

\draw [red, -{Stealth[length=3mm, width=2mm]}, shorten >=1.05cm, shorten <=1.05cm] (A1.center) to[bend left=\bendangle] (A3.center); 
\draw [blue, -{Stealth[length=3mm, width=2mm]}, shorten >=1.05cm, shorten <=1.05cm] (A1.center) to[bend right=\bendangle] (A3.center); 
\draw [green, -{Stealth[length=3mm, width=2mm]}] (A1) to (A3); 

\end{tikzpicture}
\caption{The asymptotically optimal construction for 3 colors and forbidden rainbow directed triangle.}\label{fig:directed3}
\end{figure}
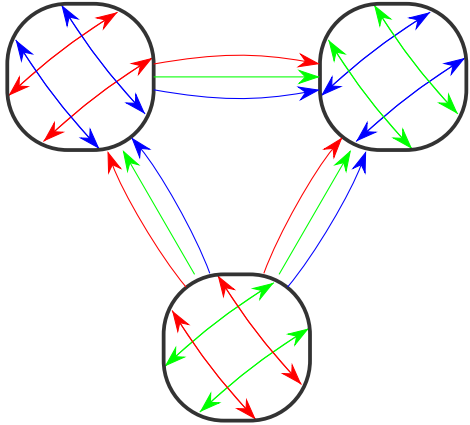

\begin{theorem}\label{cor:directed3}
Let $G_1$, $G_2$, $G_3$ be three directed graphs on a common set of $n$ vertices. If $e(G_i) \geq \frac{5}{9}n^2$ for every $i \in[3]$, then there exists a rainbow directed triangle.
\end{theorem}

In the case of a forbidden rainbow transitive triangle we show that the asymptotically optimal construction is as in the undirected setting~\cite{ADMMS20} with all edges replaced by edges directed in both ways. Formally speaking, we have two sets of $\frac{4-\sqrt{7}}{9}n$ vertices, one set of $\frac{1+2\sqrt{7}}{9}n$ vertices, each set contains a complete directed graph in two colors, while between the sets we have all edges in the color missing in the largest set (see Figure~\ref{fig:transitive3}). 

\begin{figure}[ht]
\centering
\begin{tikzpicture}[scale=.6]
\def\radius{4}
\def\baseangle{30}
\def\bendangle{10}

\pic [local bounding box=A1] at (\baseangle-1*120:\radius cm) {pic InnerGraphBig={green}{red}};
\pic [local bounding box=A2] at (\baseangle-2*120:\radius cm) {pic InnerGraph={red}{blue}};
\pic [local bounding box=A3] at (\baseangle-3*120:\radius cm) {pic InnerGraph={blue}{green}};

\draw [blue, -{Stealth[length=3mm, width=2mm]}] (A1) to[bend left=\bendangle] (A2); 
\draw [blue, -{Stealth[length=3mm, width=2mm]}, shorten >=1.05cm, shorten <=1.2cm] (A1.center) to[bend right=\bendangle] (A2.center); 
\draw [blue, -{Stealth[length=3mm, width=2mm]}] (A1) to (A2); 

\draw [blue, -{Stealth[length=3mm, width=2mm]}] (A2) to[bend right=\bendangle] (A1); 
\draw [blue, -{Stealth[length=3mm, width=2mm]}, shorten >=1.2cm, shorten <=1.05cm] (A2.center) to[bend left=\bendangle] (A1.center); 
\draw [blue, -{Stealth[length=3mm, width=2mm]}] (A2) to (A1); 

\draw [blue, -{Stealth[length=3mm, width=2mm]}] (A2) to[bend left=\bendangle] (A3); 
\draw [blue, -{Stealth[length=3mm, width=2mm]}] (A2) to[bend right=\bendangle] (A3); 
\draw [blue, -{Stealth[length=3mm, width=2mm]}] (A2) to (A3); 

\draw [blue, -{Stealth[length=3mm, width=2mm]}] (A3) to[bend right=\bendangle] (A2); 
\draw [blue, -{Stealth[length=3mm, width=2mm]}] (A3) to[bend left=\bendangle] (A2); 
\draw [blue, -{Stealth[length=3mm, width=2mm]}] (A3) to (A2); 

\draw [blue, -{Stealth[length=3mm, width=2mm]}, shorten >=1.05cm, shorten <=1.2cm] (A1.center) to[bend left=\bendangle] (A3.center); 
\draw [blue, -{Stealth[length=3mm, width=2mm]}, shorten >=1.2cm, shorten <=1.05cm] (A3.center) to[bend right=\bendangle] (A1.center); 

\draw [blue, -{Stealth[length=3mm, width=2mm]}] (A1) to[bend right=\bendangle] (A3); 
\draw [blue, -{Stealth[length=3mm, width=2mm]}] (A3) to[bend left=\bendangle] (A1);

\draw [blue, -{Stealth[length=3mm, width=2mm]}] (A1) to (A3);
\draw [blue, -{Stealth[length=3mm, width=2mm]}] (A3) to (A1);

\end{tikzpicture}
\caption{The asymptotically optimal construction for 3 colors and forbidden rainbow transitive triangle.}\label{fig:transitive3}
\end{figure}
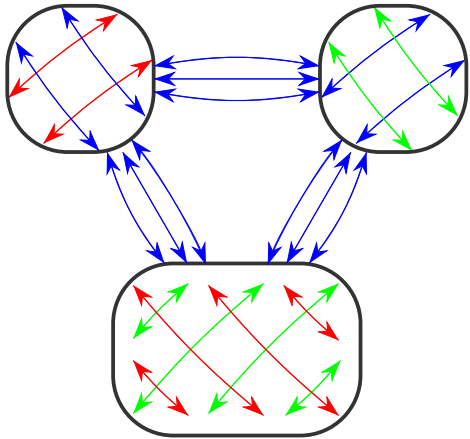

\begin{theorem}\label{cor:transitive3}
Let $G_1, G_2, G_3$ be three directed graphs on a common set of $n$ vertices.  If $e(G_i) > \left(\frac{52-4\sqrt{7}}{81}\right)n^2+ \frac{3}{2}n$ for every $i \in [3]$, then there exists a rainbow transitive triangle.
\end{theorem}

We additionally solve the analogous rainbow Tur\'an problem in the setting of oriented graphs, i.e., directed graphs without pairs of vertices connected in both directions. 
In this setting the problem is easy if a rainbow directed triangle is forbidden, because one can take each oriented graph as the same transitive tournament. 
While for a rainbow transitive triangle, we show that the optimal construction is obtained by splitting the vertex set into three equal parts with edges between the parts forming directed triangles (see Figure~\ref{fig:oriented}).

\begin{figure}[ht]
\centering
\begin{tikzpicture}[scale=0.6]
\def\radius{4}
\def\baseangle{30}

\def\bendangle{10}

\pic [local bounding box=A1] at (\baseangle-1*120:\radius cm) {pic InnerGraph={white}{white}};
\pic [local bounding box=A2] at (\baseangle-2*120:\radius cm) {pic InnerGraph={white}{white}};
\pic [local bounding box=A3] at (\baseangle-3*120:\radius cm) {pic InnerGraph={white}{white}};

\draw [red, -{Stealth[length=3mm, width=2mm]}, shorten >=1.05cm, shorten <=1.05cm] (A1.center) to[bend left=\bendangle] (A2.center); 
\draw [blue, -{Stealth[length=3mm, width=2mm]}, shorten >=1.05cm, shorten <=1.05cm] (A1.center) to[bend right=\bendangle] (A2.center); 
\draw [green, -{Stealth[length=3mm, width=2mm]}] (A1) to (A2); 

\draw [red, -{Stealth[length=3mm, width=2mm]}] (A2) to[bend left=\bendangle] (A3); 
\draw [blue, -{Stealth[length=3mm, width=2mm]}] (A2) to[bend right=\bendangle] (A3); 
\draw [green, -{Stealth[length=3mm, width=2mm]}] (A2) to (A3); 

\draw [red, -{Stealth[length=3mm, width=2mm]}, shorten >=1.05cm, shorten <=1.05cm] (A3.center) to[bend left=\bendangle] (A1.center); 
\draw [blue, -{Stealth[length=3mm, width=2mm]}, shorten >=1.05cm, shorten <=1.05cm] (A3.center) to[bend right=\bendangle] (A1.center); 
\draw [green, -{Stealth[length=3mm, width=2mm]}] (A3) to (A1); 

\end{tikzpicture}
\caption{The optimal construction for 3 oriented graphs and forbidden rainbow transitive triangle.}\label{fig:oriented}
\end{figure}
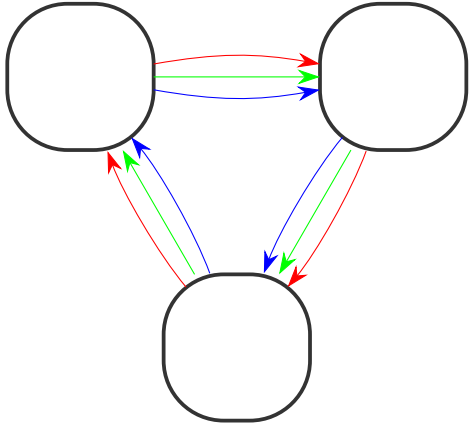

\begin{theorem}\label{cor:oriented_transitive}
Let $G_1$, $G_2$, \ldots, $G_c$ be $c\geq3$ oriented graphs on a common set of $n$ vertices. 
If $e(G_i) > \frac{1}{3}n^2$ for every $i \in [c]$, then there exists a rainbow transitive triangle.
\end{theorem}

The paper is organized as follows. 
In Section~\ref{sec:notation} we introduce the used notions.
In Section~\ref{sec:directed4+} we consider the rainbow Tur\'an problem for a directed triangle and at least 4 colors and show that the optimal asymptotic value of the maximum number of edges in each color follows from the bound on the total number of colored edges (Theorem~\ref{thm:directed4+}). 
In Section~\ref{sec:directed3} we deal with the remaining case of 3 colors for a directed triangle by showing that the optimal bound follows from the bound on the sum of the number of edges in any two colors (Theorem~\ref{thm:directed3}). 
Using similar case distinction and generalized theorems, in Section~\ref{sec:transitive4+} we solve the rainbow Tur\'an problem for a transitive triangle and at least 4 colors (Theorem~\ref{thm:transitive4+}) and in Section~\ref{sec:transitive3} we prove the optimal bound for a transitive triangle in the case of 3 colors (Theorem~\ref{thm:transitive3}). 
Finally, in Section~\ref{sec:oriented} we consider the analogous rainbow Tur\'an problems in the setting of oriented graphs and show the bound on the total number of colored edges (Theorem~\ref{thm:oriented_transitive}). 

\section{Notation and Preliminaries}\label{sec:notation}

A \emph{directed graphs} $G$ is a pair $(V(G), E(G))$, where $V(G)$ is a set of vertices and $E(G)$ is a set of ordered pairs of distinct vertices. 
In particular, $G$ does not contain loops or multiple edges. 
For shortening, we denote $e(G) = |E(G)|$. 
For two vertices $u, v \in V(G)$, we write $uv$ to denote the edge $(u,v) \in E(G)$, and refer to it as an edge \emph{from $u$ to $v$}. 
We refer jointly to edges $uv$ and $vu$ as edges \emph{between} $u$ and $v$. 
If $uv, vu \in E(G)$ we say that $u$ and $v$ are connected with a \emph{double edge}. 
While if $uv \in E(G)$ or $vu \in E(G)$, but $u$ and $v$ are not connected with a double edge, then we say that they are connected with a \emph{single edge}. 
An \emph{oriented graph} is a directed graph that contains no double edges. 
A \emph{directed triangle} is a directed graph on the vertex set $\{u,v,w\}$ with edges $uv$, $vw$ and $wu$, while a \emph{transitive triangle} is a directed graph on the vertex set $\{u,v,w\}$ with edges $uv$, $vw$ and $uw$.
We denote by $[c]$ the set of positive integers $\{1,2,\ldots,c\}$.

Having a sequence of directed graphs $G = (G_1, G_2, \ldots, G_c)$ on a common vertex set $V(G)$, we consider the edge set of each directed graph $G_i$, for $i \in [c]$, as \emph{edges} of $G$ in color~$i$. 
For two subsets $U, V \subset V(G)$ and $i \in [c]$ we denote by $e_i(U,V)$ the total number of edges $uv$ and $vu$ from $E(G_i)$ for all $u \in U$ and $v \in V$. 
Additionally, we denote by $e(U,V)$ the sum of $e_i(U,V)$ over all $i \in [c]$.
For brevity, if $U$ contains only one vertex $u$, we write $e_i(u,V)$ instead of $e_i(\{u\},V)$, similarly when $V$ contains only one vertex. 
We also denote by $e(G)$ the total number of edges in all colors, i.e., the sum of $e_i(G)$ over all $i \in [c]$.
For a directed graph~$F$, we say that $G$ contains a rainbow copy of $F$ if $V(F) \subset V(G)$ and there is a coloring $\varphi$ of the edges of $F$ into distinct colors such that $e \in E(G_{\varphi(e)})$ for every edge $e \in E(F)$.

%

In our bounds use the following lemma. It follows from Lemma 2.2 in \cite{ADMMS20}, but we add an independent proof for completeness. 

\begin{lemma}\label{lem:bipartMantel}
Let $A$ and $B$ be disjoint sets of vertices and $H$ be a graph on $A \cup B$ which does not contain triangles having vertices both from $A$ and $B$. Then
\[e(H) \leq \binom{|A|}{2}+ \binom{|B|}{2} + |A|.\]
\end{lemma}

\begin{proof}
Without loss of generality assume that $|A| \leq |B|$ as in this case the thesis is stronger.
We proceed by induction on $|A|$. 
If $|A|=0$, then $e(H) \leq \binom{|B|}{2}$ as needed. 
Now let us consider $|A|>0$ and any edge $uv \in E(H)$ with $u \in A$, $v \in B$. 
If such an edge does not exist, then $e(H) \leq \binom{|A|}{2} + \binom{|B|}{2}$.
Since edge $uv$ is not contained in any triangle, after removing vertices $u$ and $v$ we lose at most $|A|+|B|-1$ edges. The obtained graph satisfies the assumption of the lemma, so from the inductive assumption
\[e(H) \leq \binom{|A|-1}{2}+\binom{|B|-1}{2}+|A|-1 + |A|+|B|-1 = \binom{|A|}{2}+ \binom{|B|}{2} + |A|\]
as wanted.
\end{proof}

\section{Rainbow directed triangles and at least 4 colors}\label{sec:directed4+}

In this section we consider $c \geq 4$ directed graphs $G_1, G_2, \ldots, G_c$ on a common vertex set and forbid rainbow directed triangles.
The following theorem provides the optimal lower bound for the number of edges in all colors that guarantees existence of a  rainbow directed triangle, and implies the part of Theorem~\ref{cor:4+} for a rainbow directed triangle.

\begin{theorem} \label{thm:directed4+}
Let $G_1, G_2, \ldots, G_c$ be $c \geq 4$ directed graphs on a common set of $n$ vertices.
If $\sum_{i = 1}^c e(G_i) > \frac{c}{2}n^2$, then there exists a rainbow directed triangle.
\end{theorem}


\begin{proof}[Proof of Theorem~\ref{thm:directed4+}]
By contrary, assume that $G=(G_1, G_2, \ldots, G_c)$ is a sequence of directed graphs on a common set $V(G)$ that forms a counterexample with the smallest number of vertices. 
As Theorem \ref{thm:directed4+} clearly holds for directed graphs on one or two vertices, we know that $n \geq 3$. 

In a series of claims we restrict the structure of $G$ to obtain a contradiction.
We start with showing that any two vertices cannot be connected with double edges in more than one color.
Then we show that vertices connected with $c+1$ edges can create only a limited structure.
Consequently, we obtain a bound on the total number of edges contradicting the assumption on $G$.

\begin{claim}\label{cla:many_col_double_edges}
There are no vertices $u$ and $v$ such that $e_i(u, v) = 2$ and $e_j(u, v) = 2$ for some $i, j \in [c]$, $i \neq j$.
\end{claim}

\begin{proof}
By contradiction let us assume that there exist vertices $u, v$  and colors $i, j \in [c]$, $i \neq j$ such that $uv \in E(G_i) \cap E(G_j)$ and $vu \in E(G_j) \cap E(G_i)$. Without loss of generality we may assume that $i = 1$ and $j = 3$.

Fix any $k \in [c-1]$ and take an arbitrary $x \in V'= V(G) \setminus \{u, v\}$. 
If $xu \in E(G_k)$, then $vx \notin E(G_{k+1})$ as otherwise it creates a  rainbow directed triangle with an edge uv in color $1$ or $3$. 
Analogically, if $ux \in E(G_k)$, then $xv \notin E(G_{k+1})$. 
The same holds if we interchange $k$ and $k+1$. 
It implies that
\[ e_k(x, \{u, v\}) + e_{k+1}(x, \{u, v\}) \leq 4.\]
Using the same argument we have
\[ e_c(x, \{u, v\}) + e_{1}(x, \{u, v\}) \leq 4.\]
Summing up such inequalities for every $x \in V'$ and $k \in [c-1]$ we obtain
\[2 e(V', \{u, v\}) \leq 4c(n-2).\]
This implies that
\[e(G) = e(\{u,v\}) + e(\{u, v\},V') + e(V') \leq 2c + 2c(n-2) + \frac{c}{2}(n-2)^2 = \frac{c}{2}n^2,\]
which is a contradiction.
\end{proof}

This claim implies that between any two vertices there are at most $c + 1$ edges, and only one of them can be a double edge. 
We now focus on the structure created by pairs of vertices reaching this maximum. 
Let $H$ be the directed graph on $V(G)$, in which $uv \in E(H)$ if and only if $e(u,v)=c+1$ and there are at least 3 edges from $u$ to $v$. 
Note that, since $c \geq 4$, every pair of vertices connected by $c+1$ edges in $G$ is forming a directed edge in $H$, and there might be edges $uv$ and $vu$ in $H$ at the same time. 
The following claim will be useful for further considerations.

\begin{claim} \label{cla:many_col_c+1_edge}
If $uv, uw \in E(H)$ or $vu, wu \in E(H)$ for some $u, v, w \in V(G)$, then $e(v, w) \leq 2$.
\end{claim}

\begin{proof}
Assume that $uv$ and $uw$ are edges in $H$. 
This means that there exist $i \in [c]$ such that $wu \in E(G_i)$ and there are at least two edges $uv$ in $G$ in colors different than $i$. 
To avoid having a rainbow directed triangle, it implies there can be at most one edge $vw$ in $G$ (in color~$i$).
From a symmetric argument, there is at most one edge $wv$. 

The case of edges $vu$ and $wu$ in $H$ can be proven analogically.
\end{proof}

The next step is to analyze the structure of $H$. 

\begin{claim}\label{cla:many_col_cycle}
Directed graph $H$ does not contain any directed cycle.
\end{claim}

\begin{proof}
Assume the contrary and let $v_1 \ldots v_{k} v_1$ be the shortest directed cycle in $H$. 
We denote $U=\{v_1, v_2, \ldots, v_{k}\}$ and consider the indices modulo $k$.
The minimality implies that non-consecutive vertices of the cycle are not connected in $H$, so when $k>2$ in~$G$ we have 
\[e(U) \leq (c+1)k + c \left(\binom{k}{2}-k \right)  = c\binom{k}{2}+k \leq \frac{c}{2}k^2.\]
Note that the same upper bound holds if $k=2$.
If $k = n$, then the above inequality gives $e(G) \leq \frac{c}{2}n^2$, which is a contradiction. Thus, $k<n$. 

Consider an arbitrary vertex $w \in V' = V(G) \setminus U$.
If there are no edges in $H$ between $w$ and $U$, then in~$G$ we have $e(w, U) \leq ck$. 
We show that this bound holds also if there are some edges in $H$ between $w$ and the cycle. 

If $k=2$, then Claim~\ref{cla:many_col_c+1_edge} immediately implies that $e(w,U) \leq c+3 \leq ck$ as needed, while $k=3$ gives a rainbow directed triangle, so let $k \geq 4$. 
If for some $i,j \in [k]$, $i \not=j$ we have edges $v_iw$ and $wv_j$ in $H$, then, to avoid contradiction with the minimality of the cycle, $j=i+2$ and there are no edges in $H$ between $w$ and~$U \setminus \{v_i, v_j\}$. 
From Claim~\ref{cla:many_col_c+1_edge} there are at most two edges in $G$ between $w$ and $v_{i+1}$, so 
\[e(w,U) \leq 2(c+1)+2+c(k-3) = ck-c+4 \leq ck,\]
as needed. 
We are left with the case that all edges in $H$ between $w$ and $U$ have the same direction. 
Without loss of generality assume that $w$ has edges in $H$ only from $U$ and let $I = \{i \in [k] : v_iw \in E(H)\}$. 
For each $i \in I$ Claim~\ref{cla:many_col_c+1_edge} implies that there are at most two edges in $G$ between $w$ and $v_{i+1}$, so $e(w, \{v_i, v_{i+1}\}) \leq c + 3$.
This also means that denoting $m = |I|$, we have $m \leq \lfloor \frac{k}{2} \rfloor$ and
\[e(w, U) \leq (c+3)m + c(k - 2m) = ck + (3-c)m \leq ck.\]
Therefore, indeed always $e(w, U) \leq ck$.

Summing up such inequalities for all $w \in V'$ we obtain $e(V', U) \leq ck(n-k)$.
This implies that
\[e(G) = e(U) + e(U, V') + e(V') \leq \frac{c}{2}k^2 + ck(n-k) + \frac{c}{2}(n-k)^2 = \frac{c}{2}n^2,\]
which gives a contradiction.
\end{proof}

The above claim allows to show that there are no edges in $H$ at all. 

\begin{claim}\label{cla:many_col_path}
Directed graph $H$ is empty.
\end{claim}

\begin{proof}
Assume the contrary and let $v_1 v_2 \ldots v_p$ be the longest directed path in $H$. 

We start with analyzing the edges between vertices in $P = \{v_1, v_2, \ldots, v_p\}$. 
For any $i \in [p-1]$ it holds $e(v_i, v_{i+1}) = c+1$. 
For every $1 \leq i < j \leq p$ we have $v_j v_i \notin E(H)$ as otherwise it contradicts Claim~\ref{cla:many_col_cycle}. 
If $v_i v_j \notin E(H)$, then $e(v_i, v_j) \leq c$.  
While if $v_i v_j \in E(H)$ for $i < j-1$, then Claim~\ref{cla:many_col_c+1_edge} implies $e(v_i, \{v_{j-1}, v_j\}) \leq c+3$. 
Let $m =|\{(i, j) \in [p]^2: i < j-1, v_i v_j \in E(H)\}|$ and note that $m \leq \frac{1}{2} \left(\binom{p}{2} - p + 1 \right)$ and
\begin{align*}
e(P) &\leq (c+1)(p-1) + (c+3)m + c\left(\binom{p}{2} - p + 1 - 2m\right)\\
     &= c\binom{p}{2} + (3-c)m + p - 1 \leq \frac{c}{2}p^2.
\end{align*}

Observe that $p < n$ as otherwise there are fewer edges than assumed. 
Consider an arbitrary vertex $u \in V' = V \setminus P$.
If there are no edges in $H$ between $u$ and $P$, then in $H$ we have $e(u,P) \leq cp$.
If $v_iu \in E(H)$ for some $i \in [p]$, then from maximality of the path we have $i<p$, while from Claim~\ref{cla:many_col_c+1_edge} there are at most two edges between $v_{i+1}$ and $u$ in $G$, which implies $e(u, \{v_i, v_{i+1}\}) \leq c+3$.
Similarly, if $uv_j \in E(H)$ for some $j \in [p]$, then $j>1$ and $e(u, \{v_{j-1}, v_j\}) \leq c+3$.
Moreover, Claim~\ref{cla:many_col_cycle} implies that if $v_iu \in E(H)$ and $uv_j \in E(H)$, then $i<j$. 
Denote $I = \{i \in [p]: v_iu \in E(H)\}$, $J = \{j \in [p]: uv_j \in E(H)\}$ and $m = |I|+|J|$.
Note that Claim~\ref{cla:many_col_cycle} implies that $i<j$ for every $i \in I$ and $j \in J$, and so it may happen that $i+1=j-1$ only for at most one pair of $i$ and $j$. 
If $i+1\not=j-1$ for every $i \in I$ and $j \in J$, then $m \leq p/2$ and
\[e(u,P) \leq m(c+3) + (p-2m)c = cp - (c-3)m \leq cp.\]  
Otherwise, $2 \leq m \leq (p+1)/2$ and
\[e(u,P) \leq (m-1)(c+3) + (c+1) + (p-2m+1)c = cp - (c-3)(m-1) +1 \leq cp.\] 

Thus, the total number of edges in $G$ satisfies
\[ e(G) = e(P) + e(P, V') + e(V') \leq \frac{c}{2}p^2 + cp(n-p) + \frac{c}{2}(n-p)^2 = \frac{c}{2}n^2,\]
which is a contradiction with the assumed number of edges. 
\end{proof}

The above claim means that between any two vertices in $G$ there are at most $c$ edges, so 
\[e(G) \leq c\binom{n}{2} < \frac{c}{2}n^2,\]
which gives a contradiction and finishes the proof of Theorem~\ref{thm:directed4+}.
\end{proof}

\section{Rainbow directed triangles and 3 colors}\label{sec:directed3}

In case of 3 colors, Theorem~\ref{thm:directed4+} cannot hold, because if $G_1$ and $G_2$ are directed graphs with all pairs of vertices connected with double edges and $G_3$ is empty, then for large enough $n$
\[\sum_{i=1}^3 e(G_i) = 2n(n-1) > \frac{3}{2}n^2\]
and there is no rainbow directed triangle. In this case we prove the following theorem implying Theorem~\ref{cor:directed3}. 

\begin{theorem}\label{thm:directed3}
Let $G_1$, $G_2$, $G_3$ be three directed graphs on a common set of $n$ vertices. If $e(G_i)+e(G_j)\geq\frac{10}{9}n^2$ for $1 \leq i < j \leq 3$, then there exists a rainbow directed triangle.
\end{theorem}

One can slightly improve the bound in Theorem~\ref{thm:directed3} by subtracting a linear order term, but this would require assuming that $n$ is large enough to avoid constructions with a smaller quadratic term. For example in the construction depicted in Figure~\ref{fig:K_n,n} there is no rainbow triangle and in any two colors we have $n^2$ edges. This is a larger number of edges than in the construction in Figure~\ref{fig:directed3} if $n<12$.

\begin{proof}[Proof of Theorem~\ref{thm:directed3}]
By contradiction, assume that $G=(G_1, G_2, G_3)$ is a sequence of directed graphs on a common set $V(G)$ that forms a counterexample with the smallest number of vertices. 

The main idea of the proof is to split the set $V(G)$ into disjoint sets restricting the colors on edges inside and between the sets. 
Then, bound the total number of edges and the number of edges in each pair of colors in terms of the sizes of the defined set. 
This way we obtain an optimization problem on 4 variables with linear or quadratic functions, which has a unique solution giving exactly the structure depicted in Figure \ref{fig:directed3}.
We start with the following claim. 

\begin{claim}\label{cla:perfect_matching}
For every $x,y \in V(G)$ it holds $e(x,y) \leq 5$.
\end{claim}

\begin{proof}
Assuming the contrary, notice that for any $1\leq i < j \leq 3$ any vertex $v \in V(G) \setminus \{x,y\}$ satisfies $e_i(v,\{x,y\}) + e_j(v,\{x,y\}) \leq 4$. It is so, because any edge in $G_i$ between $v$ and $\{x,y\}$ excludes a possibility of an edge in $G_j$ between $v$ and $\{x,y\}$, and no two edges exclude the same edge. 
This implies that
\begin{align*}
e_i(G\setminus\{x,y\}) + e_j(G\setminus\{x,y\}) &\geq e_i(G) + e_j(G) - 4(n-2) - 4 \\
&\geq \frac{10}{9}n^2 - 4n + 4 \\
&\geq \frac{10}{9}(n-2)^2.
\end{align*}
Hence, by the minimality of $G$, $G\setminus\{x,y\}$ contains a rainbow directed triangle, so there exists a rainbow directed triangle in $G$.
\end{proof}

This claim implies that $G$ does not contain a pair of vertices connected with double edges in all three colors. Using this fact, we split $V(G)$ into many sets depicted in Figure~\ref{fig:setsXYZR}.

\vspace{-0.5cm}
\begin{figure}[ht]
\centering

\includegraphics[scale=1]{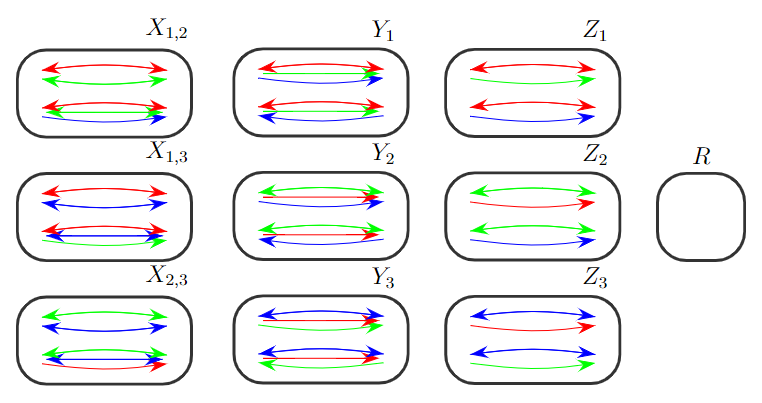}
\vspace{-0.4cm}

\caption{Split of $V(G)$ based on maximal matchings in double edges.}
\label{fig:setsXYZR}
\end{figure}

Firstly, consider a maximal matching~$X$ consisting of double edges in two colors (there might be a single edge in the third color). 
For $1\leq i < j \leq 3$ let $X_{ij}$ be the set of pairs in~$X$ with double edges in colors $i$ and $j$. Denote $x_{ij} = 2|X_{ij}|/n$ to be the fraction of vertices in~$X_{ij}$. 

Secondly, in the set $V(G) \setminus X$ consider a maximal matching~$Y$ consisting of pairs of vertices connected with 4 edges. 
From the maximality of $X$, each pair in $Y$ is connected with a double edge in exactly one color. 
For $1 \leq i \leq 3$, let $Y_i$ be the set of pairs in $Y$ with a double edge in color $i$. Denote $y_{ij} = (|Y_i| + |Y_j|)/n$ for $1\leq i < j \leq 3$.

Finally, in $V(G) \setminus (X \cup Y)$ consider a maximal matching~$Z$ consisting of pairs of vertices connected with a double edge in one color and an edge in a different color. For $1 \leq i \leq 3$, let $Z_i$ be the set of pairs in $Z$ with a double edge in color $i$. Similarly as before, denote $z_{ij} = (|Z_i| + |Z_j|)/n$ for $1\leq i < j \leq 3$.

From the maximality of $Z$, each pair of vertices in $R = V(G) \setminus (X \cup Y \cup Z)$ is connected with at most 2 edges in colors $i$ and $j$ for any $1\leq i < j \leq 3$. Let $r = |R|/n$. Notice that 
$$\sum_{1 \leq i < j \leq 3} x_{ij} + y_{ij} + z_{ij} + r = 1.$$

The following claim will be often used in order to prove useful bounds on the number of edges in particular colors.

\begin{claim}\label{cla:3oriented_bound2edge}
For any $1 \leq i < j \leq 3$, $k \in [3]$ different from $i$ and $j$, and every distinct vertices $u, v, w \in V(G)$:
\begin{itemize}[itemsep=0pt, topsep=4pt]
\item if $e_k(u,v)=1$ then $e_i(w, \{u,v\}) + e_j(w, \{u,v\}) \leq 6$, 
\item if $e_k(u,v)=2$ then $e_i(w, \{u,v\}) + e_j(w, \{u,v\}) \leq 4$.
\end{itemize}
\end{claim}

\begin{proof}
If $uv \in E(G_k)$, then the lack of rainbow directed triangles in $G$ implies that an edge $vw$ in color $i$ (or $j$) forbids an edge $wu$ in color $j$ (or $i$). Thus, at least 2 edges $vw$ and $wu$ in colors $i$ and $j$ are missing. This implies the first bullet point.
A symmetric argument in the case $vu \in E(G_k)$ implies that at least 2 edges $wv$ and $uw$ in colors $i$ and $j$ are missing, which shows the second bullet point.
\end{proof}

Note that for each pair of vertices $(u, v) \in X$ there can be at most one vertex in $V(G) \setminus X$, which is connected with both $u$ and $v$ with double edges in two colors, as otherwise it contradicts the maximality of $X$. 
Thus, by removing at most $|X|$ edges, we can guarantee that each vertex in $V(G) \setminus X$ is connected with double edges in two colors with at most one vertex in each pair in $X$.
Similarly, by removing at most $|Y|$ edges we can achieve that each vertex in $V(G) \setminus (X \cup Y)$ is connected with $4$ edges to at most one vertex in each pair in $Y$.
While by removing at most $|Z|$ edges we can guarantee that each vertex in $R$ is connected with a double edge in one color and an edge in a different color to at most one vertex in each pair in~$Z$.
Let $G'$ be the graph obtained in the above process.
Note that we cannot remove exactly $\frac{1}{2}n$ edges, so $e(G') > e(G) - \frac{1}{2}n$.

Our goal is to prove the bound
\begin{align}\label{eq:3oriented_exact2bound}
e_i(G') + e_j(G') \leq 2\binom{n}{2} + \left(x_{ij}+\frac{4}{3}y_{ij}+\frac{3}{2}z_{ij}+\frac{3}{4}r\right)^2n^2 - \left(\frac{1}{2}z_{ij}+\frac{3}{4}r\right)^2n^2+\frac{1}{2}n.    
\end{align}
This inequality translates to the following bounds on the number of edges in colors $i$ and $j$ in $G'$ between the respective sets:
\begin{itemize}[itemsep=0pt, topsep=4pt]
\item $16$ edges between different pairs of vertices in $X_{ij}$;
\item $13\frac{1}{3}$ edges between a pair of vertices in $X_{ij}$ and a pair of vertices in $Y_{i} \cup Y_{j}$;
\item $14$ edges between a pair of vertices in $X_{ij}$ and a pair of vertices in $Z_{i} \cup Z_{j}$;
\item $11\frac{5}{9}$ edges between different pairs of vertices in $Y_{i} \cup Y_{j}$;
\item $12$ edges between a pair of vertices in $Y_{i} \cup Y_{j}$ and a pair of vertices in $Z_{i} \cup Z_{j}$;
\item $12$ edges between different pairs of vertices in $Z_{i} \cup Z_{j}$;
\item $8$ edges between any other pairs of distinct vertices in $X \cup Y \cup Z$;
\item $7$ edges between a pair of vertices in $X_{ij}$ and a vertex in $R$;
\item $6$ edges between a pair of vertices in $Y_{i} \cup Y_{j}$ and a vertex in $R$;
\item $5\frac{1}{2}$ edges between a pair of vertices in $Z_{i} \cup Z_{j}$ and a vertex in $R$;
\item $4$ edges between any other pair of vertices in $X \cup Y \cup Z$ and a vertex in $R$;
\item $2$ edges between any different vertices in $R$.
\end{itemize}

Of course not all bounds are optimal. The coefficients in $(\ref{eq:3oriented_exact2bound})$ were chosen to create as simple inequality as possible, that is good enough for the proof. 

In order to show that inequality $(\ref{eq:3oriented_exact2bound})$ is satisfied, we prove separately for each bullet point that the respective bound holds.  
The bound of $16$ edges in colors $i$ and $j$ between pairs of vertices in $X_{ij}$ is obvious as there can be at most 4 such edges between any two vertices. 
Together with the definition of $G'$ it also gives the bound of $14$ edges between a pair of vertices in $X_{ij}$ and a pair of vertices in $Z_{i} \cup Z_{j}$.
Claim~\ref{cla:3oriented_bound2edge} gives that any vertex has at most 6 edges in colors $i$ and $j$ to a pair in $Y_{i} \cup Y_{j}$, which implies the bound of $12$~edges between a pair of vertices in $Y_{i} \cup Y_{j}$ and any pair in $X \cup Y \cup Z$. 
In order to fulfill the needed bound of $11$ edges between different pairs of vertices in $Y_{i} \cup Y_{j}$ one needs to additionally notice that $6$~edges in colors $i$ and $j$ between a vertex $w$ and an edge $uv$ in the third color can be achieved only when there are all $4$ edges $uw$ and $wv$ in colors $i$ and $j$, so $12$ edges between different pairs of vertices in $Y_{i} \cup Y_{j}$ requires having a double edge in both colors $i$ and $j$ contradicting the maximality of $X$.
The bound of $12$ edges between different pairs of vertices in $Z_{i} \cup Z_{j}$ also follows from the maximality of $X$ as there can be at most $3$ edges in colors $i$ and $j$ between any two vertices in $V(G)\setminus X$. 
Since any pair of vertices in $X \cup Y \cup Z$ not listed in the fist six bullet points contains a double edge in the third color, Claim~\ref{cla:3oriented_bound2edge} implies the needed bound of $8$ edges. 

Bounds between a~vertex in $R$ and a pair of vertices in $X_{ij} \cup Y_i \cup Y_j \cup Z_i \cup Z_j$ follow directly from the definition of sets $X$, $Y$, $Z$ and the definition of $G'$, while the bound of $4$ edges between a vertex in $R$ and other pairs in $X \cup Y \cup Z$ follows from Claim~\ref{cla:3oriented_bound2edge}. 
The bound between vertices in $R$ is immediate as there can be at most 2~edges in colors $i$ and $j$ between such two vertices. 
Finally, the number of edges between the two vertices in each pair in $X\cup Y \cup Z$ is properly bounded thanks to the additional linear term $\frac{1}{2}n$. 

Since $G'$ was constructed from $G$ by removing less than $\frac{1}{2}n$ edges, then also 
\[e_i(G') + e_j(G') > e_i(G)+e_i(G)-\frac{1}{2}n \geq \frac{10}{9}n^2-\frac{1}{2}n.\] 
This, together with (\ref{eq:3oriented_exact2bound}), implies that
\begin{align}\label{eq:3oriented_2bound}
\left(x_{ij}+\frac{4}{3}y_{ij}+\frac{3}{2}z_{ij}+\frac{3}{4}r\right)^2 - \left(\frac{1}{2}z_{ij}+\frac{3}{4}r\right)^2 > \frac{1}{9}.  \end{align}

Now, in a similar way, we want to bound the total number of edges in $G'$.
The maximum possible number of edges between different pairs of vertices in $X \cup Y \cup Z$ and vertices in $R$ is presented in the table below. 

\begin{center}
\begin{tabular}{|c|c|c|c|c|c|c|c|c|c|c|} 
 \hline
 & $X_{12}$ & $X_{13}$ & $X_{23}$ & $Y_1$ & $Y_2$ & $Y_3$ & $Z_1$ & $Z_2$ & $Z_3$ & $R$ \\ \hline 
 $X_{12}$ & 16 & 12 & 12 & 12 & 12 & 12 & 14 & 14 & 12 & 7 \\ \hline 
 $X_{13}$ & 12 & 16 & 12 & 12 & 12 & 12 & 14 & 12 & 14 & 7 \\ \hline 
 $X_{23}$ & 12 & 12 & 16 & 12 & 12 & 12 & 12 & 14 & 14 & 7 \\ \hline 
 $Y_1$ & 12 & 12 & 12 & 13 & 12 & 12 & 13 & 12 & 12 & 7 \\ \hline 
 $Y_2$ & 12 & 12 & 12 & 12 & 13 & 12 & 12 & 13 & 12 & 7 \\ \hline 
 $Y_3$ & 12 & 12 & 12 & 12 & 12 & 13 & 12 & 12 & 13 & 7 \\ \hline
 $Z_1$ & 14 & 14 & 12 & 13 & 12 & 12 & 12 & 12 & 12 & 6 \\ \hline 
 $Z_2$ & 14 & 12 & 14 & 12 & 13 & 12 & 12 & 12 & 12 & 6 \\ \hline 
 $Z_3$ & 12 & 14 & 14 & 12 & 12 & 13 & 12 & 12 & 12 & 6 \\ \hline
 $R$ & 7 & 7 & 7 & 7 & 7 & 7 & 6 & 6 & 6 & 3 \\ \hline 
\end{tabular}
\end{center}

Those bounds can be proven similarly as the bounds in~(\ref{eq:3oriented_exact2bound}), but some of them require a bit of case analysis. To facilitate their verification and avoid technicalities in the paper, we provide a python script available on ArXiv together with the preprint of this manuscript. 
If two pairs of vertices in $X \cup Y \cup Z$ contain a double edge in each color, then the bound of $12$ edges between them is implied by Claim~\ref{cla:3oriented_bound2edge}. 
All other bounds between a pair in $X \cup Y$ and a pair in $X\cup Y \cup Z$ are verified by the script. 
The bound of $12$ edges between pairs of vertices in $Z$ follows immediately from the maximality of $X$ and $Y$. 

From Claim~\ref{cla:3oriented_bound2edge} a vertex in $R$ can be connected to some pair in $X_{12}$ with at most $4$ edges in colors $2$ and $3$, and at most $4$ edges in colors $1$ and $3$. 
While from the definition of $G'$, it is connected with at most $7$ edges in colors $1$ and $2$. 
Altogether this implies that a vertex in~$R$ is connected to some pair in $X$ with at most $7$ edges. 
Also, a vertex in $R$ and a pair in~$Y$ are connected with at most $7$ edges from the definition of $G'$.
The remaining bounds of $6$~edges between a vertex in $R$ and a pair of vertices in $Z$, as well as $3$ edges between different vertices in $R$ are implied by the maximality of sets $X$ and $Y$. 

Not all the presented bounds can hold simultaneously. 
If a vertex in $R$ is connected with $7$ edges to two pairs in $X \cup Y$, then it reduces the number of possible edges between them.
In order to take that into account and obtain better bounds, we use Lemma~\ref{lem:bipartMantel} on an auxiliary graph. 

Let $H$ be the graph on the vertex set $X \cup Y \cup R$ with edges between different pairs in $X$, $Y$ and vertices in $R$ whenever the number of edges between them agrees with the maximum possible number of edges, but without any edges between $X_{ij}$ and $Y_i \cup Y_j$ for $1\leq i < j \leq 3$. 
Unfortunately, graph $H$ may not satisfy the assumption of Lemma~\ref{lem:bipartMantel} with partition $A = X \cup Y$ and $B = R$, but we will show that its thesis still holds. 

\begin{claim}\label{cla:3oriented_bipartMantel}
Graph $H$ satisfies 
\[e(H) \leq \binom{|X|+|Y|}{2} + \binom{|R|}{2} + |X|+|Y|.\]
\end{claim}

\begin{proof}
Assume there are two pairs in $X \cup Y$ each connected to at most one vertex in $R$. Removing them from $H$ we lose at most $2(|X|+|Y|)-1$ edges. Therefore, if the graph on the remaining vertices satisfies the needed bound on the number of edges, then also $H$ satisfies it. 
This means that we can assume that in $H$ there is at most one pair in $X \cup Y$ having less than two neighbors in $R$.
We will show that with this assumption the graph $H$ does not contain triangles having vertices both from $X \cup Y$ and $R$, thus from Lemma~\ref{lem:bipartMantel} the needed inequality holds.


If a pair $(a,b) \in X$ is connected in $H$ to a vertex $v \in R$, then $a$ or $b$ is connected to $v$ with double edges in two colors -- the same as the colors of double edges between $a$ and $b$. This implies that if $(a,b)$ is connected in $H$ to two vertices in $R$, then $a$ or $b$ has double edges in two colors to both of them as otherwise we have a contradiction with the maximality of $X$.
To avoid appearance of a rainbow directed triangle, those vertices can be connected only in these two colors, and so by at most 2 edges in $G'$. This means they are not connected in $H$. 

If a pair $(a,b) \in Y_1$ is connected in $H$ to a vertex $v \in R$, then one of its vertices, say $b$, is connected to $v$ with $4$ edges and $a$ is connected to $v$ with $3$ edges. 
Moreover, in such a case vertices $b$ and $v$ are connected with a double edge in color $1$, while from $b$ to $\{a,v\}$ there is exactly one edge in color $2$ and one edge in color $3$, as otherwise one of those two colors cannot appear between vertices $a$ and $v$ and it is not possible to have $3$ edges there.  
This implies that if a pair in $Y_1$ is connected in $H$ to two vertices in $R$, then those vertices cannot be connected with edges in colors $2$ and $3$, so they are not connected in $H$. 

Assume now that we have a triangle in $H$ consisting of a vertex in $R$ and two pairs $(a,b), (c,d) \in X_{12}$. 
This implies that the vertex in $R$ is connected with double edges in colors $1$ and $2$ to one vertex in each pair, say $b$ and $d$. 
From Claim~\ref{cla:3oriented_bound2edge} there are 16 edges between $\{a,b\}$ and $\{c,d\}$ only when all of them are in colors $1$ and $2$. 
In particular, $a$ and $c$ are connected with double edges in colors $1$ and $2$.
Now, since $(a,b)$ or $(c,d)$ is connected in $H$ to some other vertex in $R$, we obtain a contradiction with the maximality of $X$.

Consider now that we have a triangle in $H$ consisting of a vertex $v \in R$ and two pairs $(a,b)$, $(c,d)$ in $Y_1$. 
As discussed earlier, we may assume that $b$ and $v$, as well as $d$ and $v$, are connected with a double edge in color $1$, while from $b$ to $\{a,v\}$ and from $d$ to $\{c,v\}$ we have exactly one edge in color $2$ and one in color $3$. 
If vertices $b$ and $c$ are connected with $4$ edges (and so with a double edge in color $1$), from symmetry say with edge $cb$ in color $2$, then it forces edge $cv$ in color $3$. 
As a consequence, edges $bv$ and $dv$ must be in color $2$, which implies that edges $ab$ and $cd$ are also in color $2$. 
This means that there are at most $2$ edges between $a$ and $c$, and between $b$ and $d$, contradicting having $13$ edges between $\{a,b\}$ and $\{c,d\}$. 
We obtain an analogous contradiction if vertices $a$ and $d$ are connected with $4$ edges.
Assume now that vertices $b$ and $d$ are connected with $4$ edges and, from symmetry, that edge $bd$ is in color $2$.
This forces edges $dv$, $vb$, $cd$ and $ba$ in color $2$, and $db$, $vd$, $bv$, $dc$ and $ab$ in color $3$. 
It means that between $a$ and $d$, as well as between $b$ and $c$, we have only $2$ edges, contradicting having $13$ edges between $\{a,b\}$ and $\{c,d\}$. 
Therefore, we must have vertices $a$ and $c$ connected with $4$ edges. 
But then, since $(a,b)$ or $(c,d)$ is connected in $H$ with some other vertex in $R$, we have a contradiction with the maximality of~$Y$.

Assume we have a triangle in $H$ consisting of a vertex $v \in R$, a pair $(a,b) \in Y_1$ and a pair $(c,d) \in Y_2$, with $4$ edges between $b$ and $v$, and between $d$ and $v$. 
Acting exactly the same as in the previous paragraph, assuming that vertices $b$ and $c$ are connected with $4$~edges, we obtain at most $2$ edges between $a$ and $c$. 
But since between $\{a,b\}$ and $\{d,v\}$ we can have at most $12$ edges, we have at most one edge between $b$ and $d$, which contradicts having $12$ edges between $\{a,b\}$ and $\{c,d\}$.
Now, as there cannot be $3$ edges between $b$ and~$d$, we need $4$ edges between $a$ and $c$, which leads to a contradiction as before.

Consider now the possibility of a triangle in $H$ consisting of a vertex $v \in R$, a pair $(a,b) \in X_{12}$ and a pair $(c,d) \in X_{13}$. This implies that $v$ is connected with double edges in appropriate colors to one vertex in each pair, say to $b$ in colors $1$ and $2$, and to $d$ in colors $1$ and $3$. 
It means that there are no edges between $b$ and $d$.
As the bound of $12$ edges between $\{a,b\}$ and $\{c,d\}$ coming from multiple applications of Claim~\ref{cla:3oriented_bound2edge} can be achieved only when between $d$ and $\{a,b\}$ there are $4$ edges in colors $1$ and $3$, as well as $4$ edges in colors $2$ and $3$, it implies that $a$ and $d$ are connected with double edges in all colors contradicting Claim~\ref{cla:perfect_matching}.

As we do not add edges in $H$ between $X_{ij}$ and $Y_i \cup Y_j$, the final case to consider is a triangle consisting of a vertex $v \in R$, a pair $(a,b) \in X_{12}$ and a pair $(c,d) \in Y_3$. 
Similarly as in the previous cases, we can assume that $v$ is connected to $b$ with double edges in colors $1$ and $2$, and to $d$ with a double edge in color $3$. 
This means that between $b$ and $d$ there are no edges in colors $1$ and $2$. 
Similarly as in the previous paragraph, in order to have $12$ edges between $\{a,b\}$ and $\{c,d\}$ we need to have $4$ edge in colors $1$ and $2$ between $b$ and $\{c,d\}$, so $b$ and $c$ are connected with double edges in these colors. 
This implies that there are no edges in color $3$ between $c$ and $v$ contradicting that $(c,d)$ and $v$ are connected in $H$.
\end{proof}

Our goal is to show the following inequality
\begin{align}\label{eq:3oriented_exact3bound}
e(G') \leq 3\binom{n}{2} + \frac{1}{2}\sum_{1\leq i<j\leq 3}\!\!\left(x_{ij}+\frac{3}{4}y_{ij}+z_{ij}\right)^2\!\!n^2 - \frac{1}{4}\sum_{1\leq i<j\leq 3}\!\! y_{ij}z_{ij}n^2-\frac{1}{2}\sum_{1\leq i<j\leq 3}\!\!z_{ij}^2n^2+n.    
\end{align}

This inequality translates to the following bounds on the total number of edges in $G'$ between the respective sets for each $i, j \in [3]$, $i \not= j$:
\begin{itemize}[itemsep=0pt, topsep=4pt]
\item $16$ edges between different pairs of vertices in $X_{ij}$;
\item $13\frac{1}{2}$ edges between a pair of vertices in $X_{ij}$ and a pair of vertices in $Y_{i} \cup Y_{j}$;
\item $14$ edges between a pair of vertices in $X_{ij}$ and a pair of vertices in $Z_{i} \cup Z_{j}$;
\item $13\frac{1}{8}$ edges between different pairs of vertices in $Y_{i}$;
\item $12\frac{9}{16}$ edges between a pair of vertices in $Y_{i}$ and a pair of vertices in $Y_{j}$;
\item $13$ edges between a pair of vertices in $Y_{i}$ and a pair of vertices in $Z_{i}$;
\item $12\frac{1}{2}$ edges between a pair of vertices in $Y_{i}$ and a pair of vertices in $Z_{j}$;
\item $12$ edges between any other pair of distinct vertices in $X \cup Y \cup Z$;
\item $6$ edges between a pair of vertices in $X \cup Y \cup Z$ and a vertex in $R$;
\item $3$ edges between any different vertices in $R$.
\end{itemize}

Observe that the above bounds between different pairs in $X \cup Y \cup Z$ and vertices in $R$ (rounded down to the nearest integer) differ from the maximum bounds presented earlier in the following two places. Between $X \cup Y$ and $R$ we need at most $6$ edges not $7$ edges, while between $X_{ij}$ and $Y_i \cup Y_j$ we need at most $13$ edges and we have a better bound of $12$ edges. The latter pairs are exactly those pairs that we do not connect in $H$. 
Therefore, denoting by $c$ the right-hand side of $(\ref{eq:3oriented_exact3bound})$ with subtracted $\frac{1}{2}n$, we obtain that
\[e(G') \leq c + e(H) - \binom{|X|+ |Y|}{2}-\binom{|R|}{2},\]
as then all the bounds between different pairs in $X \cup Y \cup Z$ and vertices in $R$ are satisfied, while the additional linear term $\frac{1}{2}n$ in $c$ gives the remaining bound on the number of edges between the two vertices in each pair in $X \cup Y \cup Z$. 
From Claim~\ref{cla:3oriented_bipartMantel} we get that 
\[e(G') \leq c + |X|+|Y| \leq c + \frac{1}{2}n,\]
which is exactly the needed inequality $(\ref{eq:3oriented_exact3bound})$.

Since 
\[e(G') > e(G) - \frac{1}{2}n \geq \frac{5}{3}n^2 - \frac{1}{2}n,\] together with (\ref{eq:3oriented_exact3bound}), it implies that
\begin{align}\label{eq:3oriented_3bound}
\sum_{1\leq i<j\leq 3}\left(x_{ij}+\frac{3}{4}y_{ij}+z_{ij}\right)^2 - \frac{1}{2}\sum_{1\leq i<j\leq 3} y_{ij}z_{ij}-\sum_{1\leq i<j\leq 3}z_{ij}^2 > \frac{1}{3}.    
\end{align}

%

Assume, without loss of generality, that the minimal value of $x_{ij}+\frac{3}{4}y_{ij}+z_{ij}$ is attained for $i=1$ and $j=2$. 
Let $u = x_{12}+\frac{3}{4}y_{12}+z_{12}$, $y = y_{12}$ and $z = z_{12}$. 
Since
\[\sum_{1\leq i<j\leq 3}x_{ij}+\frac{3}{4}y_{ij}+z_{ij} = 1 - r - \frac{1}{2}\left(|Y_1|+|Y_2|+|Y_3|\right)/n \leq 1-r-\frac{1}{2}y\]
and the maximum value of a sum of squares of non-negative numbers with a given sum is obtained when the largest number is as large as possible, then
\[\sum_{1\leq i<j\leq 3}\left(x_{ij}+\frac{3}{4}y_{ij}+z_{ij}\right)^2 \leq 2u^2 + \left(1-r-\frac{1}{2}y-2u\right)^2.\]
Note also that 
$y_{12}z_{12} + y_{13}z_{13} + y_{23}z_{23} \geq yz$
and
\begin{align*}
z_{12}^2 + z_{13}^2 + z_{23}^2 &\geq z^2 + \frac{1}{2}\left(z_{12} + z_{13} + z_{23} -z\right)^2 \\
&= \frac{3}{2}z^2 + (z_{12} + z_{13} + z_{23})\left(\frac{1}{2}(z_{12} + z_{13} + z_{23})-z\right)\\
&= \frac{3}{2}z^2 + (z_{12} + z_{13} + z_{23})|Z_3|/n \\
&\geq \frac{3}{2}z^2.
\end{align*}

Altogether, using the above observations, from (\ref{eq:3oriented_2bound}) and (\ref{eq:3oriented_3bound}) we obtain that the following constraints are satisfied:
\begin{align*}
& \!\left(u+\frac{7}{12}y + \frac{1}{2}z+\frac{3}{4}r\right)^2 - \left(\frac{1}{2}z + \frac{3}{4}r\right)^2 > \frac{1}{9},\\
& 2u^2 + \left(1-r-\frac{1}{2}y-2u\right)^2 -\frac{1}{2}yz - \frac{3}{2}z^2 > \frac{1}{3},\\
& 3u+\frac{1}{2}y+r \leq 1,\\
& u -\frac{3}{4}y-z \geq 0,\\
& u, y, z, r \geq 0.
\end{align*}

Assuming that the above constraints are satisfied with non-strict inequalities one can check using any computer program for symbolic computations or Lagrange multipliers that the only possibility for $y$, $z$ and $r$ is to have $y = z = r = 0$, which implies $u=\frac{1}{3}$.
But this gives a contradiction with the above strict inequalities. 
\end{proof}

\section{Rainbow transitive triangles and at least 4 colors}\label{sec:transitive4+}

In this section we consider $c \geq4$ directed graphs $G_1, G_2, \ldots, G_c$ on a common vertex set and forbid rainbow transitive triangles. 
The following theorem provides the optimal lower bound for the number of edges in all  colors that guarantees existence of a rainbow transitive triangle and implies the part of Theorem~\ref{cor:4+} for a rainbow transitive triangle. 

\begin{theorem} \label{thm:transitive4+}
Let $G_1, G_2, \ldots, G_c$ be $c \geq 4$ directed graphs on a common set of $n$ vertices. 
If $\sum_{i = 1}^c e(G_i) > \frac{c}{2}n^2$, then there exists a rainbow transitive triangle.
\end{theorem}

\begin{proof}[Proof of Theorem~\ref{thm:transitive4+}]
By contradiction, assume that $G=(G_1, G_2, \ldots, G_c)$ is a sequence of directed graphs on a common set $V(G)$ that forms a counterexample with the smallest number of vertices.
In particular, $e(G) > \frac{c}{2}n^2$, there are no rainbow transitive triangles and $n \geq 3$.
The proof is obtained by showing two claims on the structure of $G$ leading to a contradiction with the number of edges. 
Firstly, we prove that in $G$ no two vertices are connected by double edges in two colors. 
Secondly, we prove that $G$ does not contain any double edge at all. 

\begin{claim}\label{cla:two_double_edges}
For any $i, j \in [c]$, $i \neq j$ there are no vertices $u$ and $v$ such that $e_i(u, v) = 2$ and $e_j(u, v) = 2$.
\end{claim}

\begin{proof}
For the sake of contradiction assume that vertices $u$ and $v$ satisfy $e_i(u, v) = 2$ and $e_j(u, v) = 2$ for some $i, j \in [c]$, $i \neq j$. 
Notice that, in order to avoid having a rainbow transitive triangle, if there is an edge between some vertex $x \in V'=V(G) \setminus \{u,v\}$ and~$u$, then between $x$ and $v$ we can have only at most $4$ edges.
Therefore, if $e(x,u)\geq1$ and $e(x,v)\geq1$, then $e(x,\{u,v\}) \leq 8 \leq 2c$. Otherwise there are no edges between $x$ and $u$ or $v$, so $e(x,\{u,v\}) \leq 2c$. 
It follows that 
\[e(G) = e(u, v) + e(\{u, v\},V') + e(V') \leq 2c + 2c(n-2) + \frac{c}{2}(n-2)^2 = \frac{c}{2}n^2,\] 
which is a contradiction.
\end{proof}

\begin{claim}\label{cla:one_double_edge}
For any $i \in [c]$ there are no vertices $u$ and $v$ such that $e_i(u, v) = 2$.
\end{claim}

\begin{proof}
Suppose for a contradiction that there exist $i \in [c]$ and vertices $u$ and $v$ satisfying $e_i(u,v)=2$.
For any vertex $x \in V'=V(G) \setminus \{u,v\}$, if $e_j(x,u)\geq1$ and $e_j(x,v)\geq1$ for some $j \in [c] \setminus \{i\}$, then between $x$ and $\{u,v\}$ we have only edges in colors $i$ and $j$, so $e(x,\{u,v\}) \leq 6$ from Claim~\ref{cla:two_double_edges}. 
Otherwise, $e(x,\{u,v\}) \leq c+3$, because from Claim~\ref{cla:two_double_edges} we have only at most two double edges between $x$ and $\{u,v\}$, and by removing one edge in each of them we have at most 1 edge in each $G_j$ for $j \in [c]\setminus\{i\}$ and 2 edges in $G_i$.
It follows that
\[e(G) = e(u, v) + e(\{u, v\},V') + e(V') \leq c+1 + (c+3)(n-2) + \frac{c}{2}(n-2)^2 < \frac{c}{2}n^2,\] 
which is a contradiction.    
\end{proof}

The above claim implies that between any two vertices in $G$ there are at most $c$ edges, so 
\[e(G) \leq c\binom{n}{2} < \frac{c}{2}n^2\]
contradicting the definition of $G$.
\end{proof}

\section{Rainbow transitive triangles and 3 colors}\label{sec:transitive3}

Similarly as in the case of the forbidden rainbow directed triangle, in case of 3 colors Theorem~\ref{thm:transitive4+} does not hold. In this case we prove the following theorem implying Theorem~\ref{cor:transitive3}. 

\begin{theorem}\label{thm:transitive3}
Let $G_1, G_2, G_3$ be three directed graphs on a common set of $n$ vertices. If for every $1 \leq i < j \leq 3$ it holds $e(G_i)+e(G_j)>\left(\frac{104-8\sqrt{7}}{81}\right)n^2+3n$, then there exists a rainbow transitive triangle.
\end{theorem}

In the proof we make use of the analogous theorem in the undirected setting. 

\begin{theorem}[Aharoni et al. \cite{ADMMS20}]\label{thm:undirected3}
Let $H_1, H_2, H_3$ be graphs on a common set of $n$ vertices. If for every $1 \leq i < j \leq 3$ it holds $e(H_i)+e(H_j)>\left(\frac{52-4\sqrt{7}}{81}\right)n^2+\frac{3}{2}n$, then there exists a rainbow triangle.
\end{theorem}

Since here we forbid only rainbow transitive triangles and allow rainbow directed triangles, Theorem~\ref{thm:transitive3} does not follow straightforwardly from Theorem~\ref{thm:undirected3}.
Nevertheless, we prove that a hypothetical counterexample with the maximum number of edges does not contain rainbow directed triangles and apply Theorem~\ref{thm:undirected3}.

\begin{proof}[Proof of Theorem~\ref{thm:transitive3}]
By contradiction, assume that $G=(G_1, G_2, G_3)$ is a sequence of directed graphs on a common set $V(G)$ that forms a counterexample with the smallest number of vertices, and among such, with the largest number of edges.

\begin{claim}
For every $i \in [3]$ and $u,v \in V(G)$ if $uv \in E(G_i)$ then $vu \in E(G_i)$. 
\end{claim}

\begin{proof}
Assume by contradiction that the claim does not hold and for $i \in [3]$ let $s_i$ be the number of unordered pairs $u, v \in V(G)$ connected with a single edge in color $i$. 
Without loss of generality assume that $s_1 \leq s_2 \leq s_3$. 
From $G$ we create a sequence $G'=(G'_1, G'_2, G'_3)$ of directed graphs in the following way. 
For each $i\in[3]$ each pair of vertices connected with a double edge in $G_i$ is also connected with a double edge in $G'_i$. 
All $s_1$ single edges in color $1$ are removed, while all the remaining $s_2 + s_3$ pairs of vertices connected with a single edge in colors $2$ or $3$ become connected with a double edge in $G'_2$ or $G'_3$, respectively. 
Note that for every $1 \leq i < j \leq 3$ it holds $e(G'_i)+e(G'_j) \geq e(G_i)+e(G_j)$ and also $e(G') > e(G)$. 

It remains to show that $G'$ does not contain a rainbow transitive triangle as then it forms a counterexample to Theorem~\ref{thm:transitive3} with a larger number of edges than $G$ contradicting its maximality. 
Thus, assume that vertices $u,v,w$ form a rainbow transitive triangle in $G'$.
As all pairs of vertices in $G'$ may only be connected with a double edge and no edge in color $1$ was added while creating $G'_1$, vertices $u,v,w$ are forming a triangle in $G$ with a double edge in color $1$ and at least one edge in colors $2$ and $3$. 
Regardless of the orientation of those two edges, we have a forbidden rainbow transitive triangle in $G$. 
\end{proof}

This claim shows that for every $i \in [3]$ between every two vertices in $G$ we have either 0 or 2 edges in color $i$. For each $i \in [3]$ let $H_i$ be the graph on $V(G)$ obtained by joining those pairs of vertices which are connected with a double edge in $G_i$. 
Note that for every $1 \leq i < j \leq 3$ it holds \[e(H_i)+e(H_j) \geq \frac{1}{2}e(G_i)+\frac{1}{2}e(G_j) > \left(\frac{52-4\sqrt{7}}{81}\right)n^2+\frac{3}{2}n.\]
Thus, from Theorem~\ref{thm:undirected3} there is a rainbow triangle in H. 
This implies that in $G$ we have a rainbow triangle formed by double edges, which contains a rainbow transitive triangle contradicting the assumption on $G$.
\end{proof}

\section{Oriented graphs without rainbow triangles}\label{sec:oriented}

In this section we consider the setting when $G_1, G_2, \ldots, G_c$ are  oriented graphs on a common set of $n$ vertices. 
When a rainbow directed triangle is forbidden, in contrary to the setting of directed graphs presented in Sections \ref{sec:directed4+} and \ref{sec:directed3}, in this setting the problem becomes trivial.
It is so, because if each $G_i$ for $i \in [c]$ is the same transitive tournament, then we have the maximum possible number of edges in each oriented graph and there is no rainbow directed triangle. 


In the other case, when a rainbow transitive triangle is forbidden, we prove the following theorem implying Theorem~\ref{cor:oriented_transitive}.

\begin{theorem}\label{thm:oriented_transitive}
Let $G_1$, $G_2$, \ldots, $G_c$ be $c\geq3$ oriented graphs on a common set of $n$ vertices. 
If $\sum_{i = 1}^c e(G_i) > \frac{c}{3}n^2$, then there exists a rainbow transitive triangle.
\end{theorem}


\begin{proof}[Proof of Theorem~\ref{thm:oriented_transitive}]
Let $G = (G_1, G_2, \ldots, G_c)$ be a sequence of oriented graphs on a common set $V(G)$ that forms a hypothetical counterexample with the smallest number of vertices.
Obviously, $n\geq 4$ as otherwise the assumption on the number of edges cannot hold.

We say that three vertices $u$, $v$ and $w$ form a \emph{thick path} if there are at least 3 edges from $u$ to $v$ and at least 3 edges from $v$ to $w$. 
We start with showing that $G$ does not contain a thick path. 

\begin{claim}\label{cla:path}
$G$ does not contain a thick path. 
\end{claim}

\begin{proof}
Assume by contradiction that $G$ contains vertices $u$, $v$ and $w$ forming a thick path.
Consider any vertex $x \in V(G) \setminus \{u,v,w\}$.
If there exists an edge from $x$ to $v$ in $G_i$ for some $i \in [c]$, then, in order to avoid having a rainbow transitive triangle, there cannot be any edge between $x$ and $u$ in a color different than $i$. 
Thus, the number of edges from $x$ to $v$ and between $x$ and $u$ is bounded by $c$. 
By a symmetric argument, the number of edges from $v$ to $x$ and between $x$ and $w$ is bounded by $c$.
Therefore, there are at most $2c$ edges between $x$ and the set $\{u,v,w\}$.
Denoting $V'=V(G) \setminus \{u,v,w\}$ we have
$$ e(G) = e(\{u,v,w\}) + e(\{u,v,w\},V') + e(V') \leq 3c + 2c(n-3) + \frac{c}{3}(n-3)^2 = \frac{c}{3}n^2,$$
which gives a contradiction with the assumed number of edges in $G$. 
\end{proof}

This property allows us to further restrict the structure of $G$.

\begin{claim}\label{cla:edge}
For every $u, v \in V(G)$ there are at most 2 edges between $u$ and $v$. 
\end{claim}

\begin{proof}
Assume by contradiction that $G$ contains two vertices $u,v$ with at least 3 edges between them. 
We will show that for any vertex $x \in V(G) \setminus \{u,v\}$ there are at most $c+1$ edges between $x$ and $\{u,v\}$.

Notice that if $xu \in E(G_i)$ and $xv \in E(G_j)$ for $i \not= j$, then we have a rainbow transitive triangle. 
Thus, there are no edges $xu$, no edges $xv$, or exactly 1 edge $xu$ and 1 edge $xv$. 
A symmetric argument holds for edges from $u$ and $v$ to $x$. 
Since we work in the setting of oriented graphs, the bound of $c+1$ edges between $x$ and $\{u,v\}$ may not hold only when there are no edges $xu$ and no edges $vx$ (or in the symmetric case).
In order to avoid creating a thick path forbidden by Claim~\ref{cla:path} and have more than $c+1$ edges between $x$ and $\{u,v\}$, we may assume that there are $c$ edges from $u$ to $x$ and $2$ edges from $v$ to $x$. 
If there is an edge $uv$, then we have a rainbow transitive triangle. 
Otherwise, there are at least three edges from $v$ to $u$ and we obtain a contradiction with Claim~\ref{cla:path}.
Therefore, indeed we have at most $c+1$ edges between $x$ and $\{u,v\}$.
Denoting $V'=V(G) \setminus \{u,v\}$ this implies that
\begin{align*}
e(G) &= e(\{u,v\}) + e(\{u,v\},V') + e(V') \leq c + (c+1)(n-2) + \frac{c}{3}(n-2)^2 \\
     &= \frac{c}{3}n^2 - \left(\frac{c}{3}-1\right)(n-1) - 1 < \frac{c}{3}n^2,
\end{align*}
which gives a contradiction with the assumed number of edges in $G$.
\end{proof}

Claim~\ref{cla:edge} implies that the number of edges in $G$ is at most $2 \binom{n}{2} < n^2 \leq \frac{c}{3}n^2$, which is a contradiction proving Theorem~\ref{thm:oriented_transitive}.
\end{proof}


\end{document}